\documentclass[12pt]{amsart}
\usepackage[OT2,T1]{fontenc}
\usepackage{amsmath,amsfonts,amsthm,amssymb,mathrsfs,
amsxtra,amscd,latexsym,color}
\usepackage[hmargin=2cm,vmargin=3cm]{geometry}
\DeclareSymbolFont{cyrletters}{OT2}{wncyr}{m}{n}
\DeclareMathSymbol{\Sha}{\mathalpha}{cyrletters}{"58}

 \newtheorem{thm}{Theorem}[section]
 
 \newtheorem{cor}[thm]{Corollary}
 
 \newtheorem{lem}[thm]{Lemma}
 \newtheorem{prop}[thm]{Proposition}

 \newtheorem{rmk}[thm]{Remark}
 
 \theoremstyle{definition}
 \theoremstyle{remark}
 \theoremstyle{question}
 \numberwithin{equation}{section}

\newcommand{\sm}{\left(\begin{smallmatrix}}
\newcommand{\esm}{\end{smallmatrix}\right)}
\newcommand{\mat}{\left(\begin{matrix}}
\newcommand{\emat}{\end{matrix}\right)}

\def\HH{\mathbb{H}}

\def\ZZ{\mathbb{Z}}

\def\m{\mathrm{mod}}

\def\SL{\mathrm{SL}}

\begin{document}

\title{Congruences involving the $U_{\ell}$ operator for weakly holomorphic modular forms}


\author{Dohoon Choi}

\author{Subong Lim}

\address{Department of Mathematics, Korea University, 145 Anam-ro, Seongbuk-gu, Seoul 02841, Republic of Korea}
\email{dohoonchoi@korea.ac.kr}	

\address{Department of Mathematics Education, Sungkyunkwan University, Jongno-gu, Seoul 03063, Republic of Korea}
\email{subong@skku.edu}

\subjclass[2010]{11F33, 11F37}

\thanks{Keywords: congruence, modular forms of half-integral weight, trace of singular moduli  }


\begin{abstract}
Let $\lambda$ be an integer, and  $f(z)=\sum_{n\gg-\infty} a(n)q^n$ be a weakly holomorphic modular form of weight $\lambda+\frac 12$ on $\Gamma_0(4)$ with integral coefficients.
Let $\ell\geq 5$ be a prime.
Assume that the constant term $a(0)$ is not zero modulo $\ell$.
Further, assume that, for some positive integer $m$, the Fourier expansion of $(f|U_{\ell^m})(z) = \sum_{n=0}^\infty b(n)q^n$ has the form
\[
(f|U_{\ell^m})(z) \equiv  b(0) + \sum_{i=1}^{t}\sum_{n=1}^{\infty} b(d_i n^2) q^{d_i n^2} \pmod{\ell},
\]
where $d_1, \ldots, d_t$ are square-free positive integers, and  the operator $U_\ell$ on formal power series  is defined by
\[
\left( \sum_{n=0}^\infty a(n)q^n \right) \bigg| U_\ell = \sum_{n=0}^\infty a(\ell n)q^n.
\]
Then, $\lambda \equiv 0 \pmod{\frac{\ell-1}{2}}$. Moreover, if $\tilde{f}$ denotes the coefficient-wise reduction of $f$ modulo $\ell$, then we have
\[
  \biggl\{ \lim_{m \rightarrow \infty} \tilde{f}|U_{\ell^{2m}},  \lim_{m \rightarrow \infty} \tilde{f}|U_{\ell^{2m+1}} \biggr\} = \biggl\{ a(0)\theta(z), a(0)\theta^\ell(z) \in \mathbb{F}_{\ell}[[q]] \biggr\},
\]
where $\theta(z)$ is the Jacobi theta function defined by $\theta(z) = \sum_{n\in\ZZ} q^{n^2}$.
By using this result, we obtain the distribution of the Fourier coefficients of weakly holomorphic modular forms in congruence
classes.
This applies to the congruence properties for traces of singular moduli.
\end{abstract}

\maketitle


\section{Introduction}
The distribution of the coefficients of half-integral weight modular forms in congruence classes is related to the study of the congruence properties for various objects such as the algebraic parts of the central critical values of modular $L$-functions, orders of Tate-Shafarevich groups of elliptic curves, the number of partitions of a positive integer, and so on. With such diverse applications, this subject has been studied in the works of Bruinier \cite{B2}, Bruinier and Ono \cite{BO}, Ono and Skinner \cite{OS}, and Ahlgren and Boylan \cite{AB, AB1}.

 Vign{\'e}ras \cite{Vig} proved that if $f$ is a modular form of half-integral weight on $\Gamma_1(4N)$ whose coefficients are supported on finitely many square classes, then $f$ is a linear combination of Shimura theta series (a different proof of this result was given by Bruinier \cite{B1}). For a prime $\ell$, the mod $\ell$ extension of a characteristic zero theorem of Vign{\'e}ras can be considered as a classification of the modular forms $f(z) = \sum_{n=0}^\infty a(n)q^n$ of half-integral weight having  Fourier expansion of the form
\begin{equation}\label{thetaform}
f(z) \equiv a(0) + \sum_{i=1}^{t}\sum_{n=1}^{\infty} a(d_i n^2) q^{d_i n^2} \pmod{\ell},
\end{equation}
where  $q=e^{2\pi iz}$ for a complex number $z$ in the complex upper half plane $\HH$ and $d_1, \ldots, d_t$ are square-free positive integers.
Modular forms of half-integral weight having  Fourier expansion of the form (\ref{thetaform}) play important roles in proving many of the above theorems.

In this vein, this paper studies  $f|U_{\ell^m}$ with Fourier expansion of the form (\ref{thetaform}) for a weakly holomorphic modular form $f$ on $\Gamma_0(4)$,
and then obtains  the distribution of the Fourier coefficients of weakly holomorphic modular forms in congruence classes.
Here, the operator $U_\ell$ on formal power series  is defined by
\[
\left( \sum_{n=0}^\infty a(n)q^n \right) \bigg| U_\ell = \sum_{n=0}^\infty a(\ell n)q^n.
\]
Further, we apply this result to the congruence properties for traces of singular moduli.

Let $M^!_{k}$ denote the space of  weakly holomorphic modular forms of  weight $k$ on $\Gamma_0(4)$ with  trivial character. Throughout this paper,  we  assume that $\ell$ is a prime larger than or equal to $5$. For a weakly holomorphic modular form $f=\sum_{n\gg-\infty} a(n)q^n$ with integral coefficients, we denote by $\tilde{f}$ the coefficient-wise reduction of $f$ modulo $\ell$.

In the following theorem, we prove a sufficient condition of weight $\lambda+\frac12$ for weakly holomorphic modular forms $f\in M^!_{\lambda+\frac12}$ with $\lambda\in\ZZ$ such that
$f|U_{\ell^m}$ is supported on finitely many square classes modulo $\ell$ for some positive integer $m$, and then we show that, for such a form $f$, the limits $\lim_{m\to\infty} \tilde{f}|U_{\ell^{2m}}$ and  $\lim_{m\to\infty} \tilde{f}|U_{\ell^{2m+1}}$ are convergent to $a(0)\theta$ or $a(0)\theta^{\ell}$ in $\mathbb{F}_{\ell}[[q]]$,
where $\theta$ is the Jacobi theta function defined by $\theta(z) = \sum_{n\in\ZZ} q^{n^2}$.

\begin{thm} \label{main1}
Let $\ell\geq 5$ be a prime and $\lambda$ be an integer.
Let $f(z)=\sum_{n\gg-\infty} a(n)q^n\in M^!_{\lambda+\frac12}$ be a weakly holomorphic modular form with integral  coefficients.
Assume that the constant term $a(0)$ is not zero modulo $\ell$.
Further, assume that, for some positive integer $m$, the  Fourier expansion  of $f|U_{\ell^m}$ has the form (\ref{thetaform}).
Then, $\lambda \equiv 0 \pmod{\frac{\ell-1}{2}}$. Moreover, we have
\[
\biggl\{ \lim_{n\to\infty} \tilde{f}|U_{\ell^{2n}}, \lim_{n\to\infty} \tilde{f}|U_{\ell^{2n+1}} \biggr\} = \biggl\{ a(0)\theta, a(0)\theta^{\ell} \in \mathbb{F}_{\ell}[[q]] \biggr\}.
\]
\end{thm}

Let $M$ be a positive integer. With the same notation as in Theorem \ref{main1}, we say that the Fourier coefficients of $f$ are {\it well-distributed modulo $M$} if for every integer $r$, we have
\begin{equation*}
\#\{ 0 \leq n \leq X \; | \;  a(n)\equiv r\ (\mathrm{mod}\ M) \} \gg_{r,M}
\left\{
\begin{array}{ll}
{\sqrt{X}}/{\log X}  & \text{ if }  r \not \equiv 0 \pmod{M},   \\
X                        & \text{ if } r \equiv 0 \pmod{M}.
\end{array}
\right.
\end{equation*}
Bruinier and Ono \cite{BO} proved, for a prime $\ell \geq5$, that if  the cusp form $f$ does not have  Fourier expansion of the form (\ref{thetaform}), then, for each positive integer $j$, the Fourier coefficients of $f$ are well-distributed modulo $\ell^j$. Thus, the classification of modular forms having the form (\ref{thetaform}) can be applied to  the study of the distribution of the coefficients of modular forms in congruence classes modular $\ell^j$ (for example, see \cite{AB} and \cite{C3}). In this vein, we prove the following theorem.

\begin{thm} \label{main2}
Let $\ell \geq 5$ be a prime and $\lambda$ be an integer.
Let $f(z)=\sum_{n\gg-\infty} a(n)q^n\in M^!_{\lambda+\frac12}$ be a weakly holomorphic modular form with integral coefficients. Suppose that the constant term $a(0)$ is not zero modulo $\ell$. If $\lambda \not \equiv 0 \pmod{\frac{\ell-1}{2}}$, then  for each positive integer $j$, the Fourier coefficients of $f$ are well-distributed modulo $\ell^j$.
\end{thm}

\begin{rmk}
It was proved in \cite{C2} that if the Fourier coefficients of a weakly holomorphic modular form of weight $\lambda+\frac12$ on $\Gamma_0(4N)$ with non-zero constant term modulo $\ell$ are not well-distributed modulo $\ell^j$, then $\lambda \equiv 0\ \text{or}\ 1 \pmod{\frac{\ell-1}{2}}$. Thus, Theorem \ref{main2} improves this result in the case of $\Gamma_0(4)$.
\end{rmk}

Let $J(z):=j(z)-744$ be the normalized Hauptmodul for $\mathrm{SL}_2(\mathbb{Z})$, where $j(z)$ is the modular $j$-invariant.
Let $d$ be a positive integer with $d\equiv 0$ or $3\ (\m\ 4)$. We denote by $\mathcal{Q}_d$ the set of positive definite binary quadratic forms  of discriminant $-d$ with the usual action of $\SL_2(\ZZ)$. For a binary form $Q \in \mathcal{Q}_d$, let $\alpha_Q$ be the unique root of $Q(x,1)$ in $\HH$. The values $J(\alpha_Q)$ are called singular moduli, which play important roles in number theory. For example, singular moduli generate  Hilbert class fields of imaginary quadratic fields.
We define the modular trace function ${\bf t}(d)$ by
\begin{equation}
{\bf{t}}(d) := \sum_{Q\in\mathcal{Q}_d/\SL_2(\ZZ)} \frac{1}{\omega_Q}J(\alpha_Q),
\end{equation}
where $\omega_Q = |\mathrm{PSL}_2(\ZZ)_Q|$.
In \cite{Zag}, Zagier showed that
\[
q^{-1} -2 - \sum_{d>0} {\bf t}(d)q^d
\]
is a weakly holomorphic modular form of weight $\frac32$ on $\Gamma_0(4)$.
The congruence properties for traces of singular moduli were studied in several articles such as \cite{AO}, \cite{Bo}, \cite{G}, and \cite{T}.
In the following corollary, we consider the distribution of traces of singular moduli modulo powers of a prime $\ell\geq 5$.

\begin{cor}
If $\ell \geq 5$ is a prime, then  for each positive integer $j$
$$\#\{ 0 \leq d \leq X \; | \; {\bf t} (d)\equiv r \pmod{\ell^j} \} \gg_{r, \ell^j}
\left\{
\begin{array}{ll}
{\sqrt{X}}/{\log X}  & \text{ if }  r \not \equiv 0 \pmod{\ell^j},   \\
X                        & \text{ if } r \equiv 0 \pmod{\ell^j}.
\end{array}
\right.$$
\end{cor}

The remainder of the paper is organized as follows. In Section \ref{section2}, we introduce some preliminaries for the filtration and  distribution of the Fourier coefficients of modular forms of half-integral weight. In Section \ref{section3}, we prove Theorems \ref{main1} and \ref{main2}.

\section{Preliminaries} \label{section2}
In this section, we introduce some notions and properties concerning  mod $\ell$ (weakly holomorphic) modular forms and prove some lemmas needed to prove the main theorems. Especially, for $\Gamma_0(4)$, we extend the properties for the filtration of a modular form of integral weight to a modular form of half-integral weight.

\subsection{Modular forms}
In this subsection, we fix some notations about modular forms.
If $d$ is an odd prime, then let $\left(\frac cd\right)$ be the usual Legendre symbol.
For a positive odd integer $d$, we define $\left(\frac cd\right)$ multiplicatively.
For a negative odd integer $d$, let
\begin{equation*}
\left(\frac cd\right) =
\begin{cases}
\left(\frac{c}{|d|}\right) & \text{if $d<0$ and $c>0$},\\
-\left(\frac{c}{|d|}\right) & \text{if $d<0$ and $c<0$}.
\end{cases}
\end{equation*}
Furthermore, let $\left(\frac{0}{\pm 1}\right) = 1$.
For an odd integer $d$, we define $\epsilon_d$ by
\begin{equation*}
\epsilon_d =
\begin{cases}
1 & \text{if $d\equiv 1 \pmod{4}$},\\
i & \text{if $d \equiv 3  \pmod{4}$}.
\end{cases}
\end{equation*}
In this paper, we use the convention: $z = |z|e^{i\mathrm{arg}(z)},\ -\pi < \mathrm{arg}(z) \leq \pi$.
For a function $f$ on $\HH$ and $\gamma = \sm a&b\\c&d\esm\in \Gamma_0(4)$, we define the slash operator
\begin{equation*}
(f|_k \gamma)(z) =
\begin{cases}
(cz+d)^{-k} f(\gamma z) & \text{if $k \in 2\ZZ$},\\
\chi_{-4}(d)^{-1}(cz+d)^{-k}f(\gamma z) & \text{if $k\in 2\ZZ+1$},\\
\left(\frac cd\right)^{-2k}\epsilon_d^{2k}(cz+d)^{-k}f(\gamma z) & \text{if $k\in \ZZ+\frac12$},
\end{cases}
\end{equation*}
where $\chi_{-4}(d) = \left(\frac{-4}{d}\right)$.

Let $k\in \frac12\ZZ$ and $N$ be a positive integer.
Let $\chi$ be a Dirichlet character modulo $4N$.
A function $f$ on $\HH$ is called a weakly holomorphic modular form of weight $k$ on $\Gamma_0(4N)$ with character $\chi$ if it is holomorphic on $\HH$, meromorphic at the cusps, and satisfies
\[
f|_k \gamma = \chi(d)f
\]
for all $\gamma = \sm a&b\\c&d\esm \in \Gamma_0(4N)$.
We say that $f$ is a holomorphic modular form if it is holomorphic at the cusps, and a cusp form if it vanishes at the cusps.
We denote by $M_k$ the space of holomorphic modular forms of weight $k$ on $\Gamma_0(4)$ with  trivial character.

\subsection{Filtration for mod $\ell$ modular forms of half-integral weight}
Let $N$ be a positive integer with $\ell \nmid N$. The theory of mod $\ell$ modular forms of integral weight was established by Serre \cite{Ser} and Swinnerton-Dyer \cite{Swi} for modular forms of level $1$.  Their results were generalized to modular forms of higher level by Katz \cite{Kat} and Gross \cite{Gro}. For a modular form $f$ on $\Gamma_0(N)$ with $\ell$-integral coefficients, the filtration $\omega_{\ell}(f)$ of $f$ modulo $\ell$ is defined to be the infimum of  $k \in \ZZ_{\geq0}$ such that there is a modular form $g$ of weight $k$ on $\Gamma_0(N)$ with $\ell$-integral rational Fourier coefficients satisfying $f \equiv g \pmod{\ell}$. The theory of mod $\ell$ modular forms of integral weight gives several properties for the filtration of mod $\ell$ modular forms.

In the following theorem, we summarize some well-known properties of the filtration for modular forms of integral weight.
For the details, we refer to \cite[Section 1]{Joch}.

\begin{thm} \label{property_integral}
Let $\ell\geq 5$ be a prime and $k$ be an integer. Let $f$ be a modular form in $M_k$.
Then, we have the following.
\begin{enumerate}
\item $\omega_{\ell}(f) \equiv k \pmod{\ell-1}$.

\item $\omega_{\ell}(f^m) = m \cdot \omega_{\ell}(f)$ for a positive integer $m$.

\item $\omega(f|U_{\ell}) \equiv \omega(f) \pmod{\ell-1}$.

\item  $\omega_{\ell}(f|U_\ell) \leq \ell + \frac{\omega_{\ell}(f)-1}{\ell}$.
\end{enumerate}
\end{thm}

In this subsection, for $\Gamma_0(4)$, we extend the properties for the filtration of a modular form of integral weight to a modular form of half-integral weight. Suppose that $f$ is a modular form of weight in  $\ZZ+\frac12$ on $\Gamma_0(4)$ with $\ell$-integral rational Fourier coefficients.
The filtration $\omega_{\ell}(f)$ of $f$ modulo $\ell$ is defined to be the infimum of  $k \in \frac12 \ZZ_{\geq0}$ such that there is a modular form $g \in M_k$ with $\ell$-integral rational Fourier coefficients satisfying $f \equiv g \pmod{\ell}$. Then, the filtration satisfies some properties that are analogous to the properties of filtration for modular forms of integral weight.

\begin{prop} \label{filtration}
Let $\ell\geq 5$ be a prime, and $k, k_1$, and $k_2$ be elements of $\frac12\ZZ$.
\begin{enumerate}
\item For a modular form $f\in M_{k}$ and a positive integer $m$,
\[
\omega_\ell(f \theta^m) = \omega_\ell(f) + \frac{m}{2}.
\]


\item For a modular form $f\in M_k$ and a positive integer $m$,
\[
\omega_\ell(f^m) = m \cdot \omega_\ell(f).
\]

\item For a modular form $f\in M_{\lambda+\frac12}$,  there is a modular form $g$ on $\Gamma_0(4)$ such that
$g \equiv f|U_{\ell} \pmod{\ell}$. 
Moreover, we have
\[
\omega_{\ell}(f|U_{\ell}) - \omega_{\ell}(f) \equiv 0 \pmod{\frac{\ell-1}{2}}.
\]

\item For a modular form $f\in M_{\lambda+\frac12}$,
\[
\lambda \equiv \omega_{\ell}(f) - \frac12 \pmod{\ell-1}.
\]
\end{enumerate}
\end{prop}

\begin{rmk}
Tupan \cite{Tu} studied the ring structure of mod $\ell$ modular forms on $\Gamma_1(4)$. One can give another proof of Proposition \ref{filtration} by using the method of Swinnerton-Dyer \cite{Swi} and the result of Tupan.
\end{rmk}

\begin{proof}
(1) By the definition of $\omega_\ell(f)$, there is a modular form $\hat{f}$ in $M_{\omega_\ell(f)}$ with $\ell$-integral rational Fourier coefficients such that
\[
f \equiv \hat{f} \pmod{\ell}.
\]
Therefore, we see that
\[
f\theta^m \equiv \hat{f} \theta^m \pmod{\ell}
\]
and the weight of $\hat{f} \theta^m$ is $\omega_\ell(f) + \frac{m}{2}$.
This implies that
\[
\omega_\ell(f\theta^m) \leq \omega_\ell(f) + \frac m2.
\]

On the other hand, we let $k_0 = \omega_\ell(f\theta^m)$.
Then, there is a modular form $F \in M_{k_0}$ with $\ell$-integral rational Fourier coefficients such that
\begin{equation} \label{Fdfn}
F \equiv f \theta^m \pmod{\ell}.
\end{equation}
We define
\[
f_0 = \frac{F}{\theta^m},
\]
which is a weakly holomorphic modular form with $\ell$-integral rational coefficients.
It has a possible pole only at the cusp $\frac12$.
Let us note that $f_0 \equiv f \pmod{\ell}$ by (\ref{Fdfn}). By the q-expansion principle,  the  Fourier expansion of $f_0^2$ at $1/2$ is congruent to that of $f^2$ at $1/2$ modulo $\ell$ (for example, see Ramark 12.3.5 in \cite{Di-Im}). Since $\ell>2$, $f_0$ and $f$ have the same Fourier expansion at $1/2$ modulo $\ell$.  Thus, we see that at the cusp $\frac12$, $f_0$ has  Fourier expansion of the form
\[
\ell \cdot \sum_{n<0} a(n)q^n + O(1)
\]
for some $\ell$-integral coefficients $a(n)$.

Note that from Ligozat \cite{Lig} we find that the function
\begin{equation} \label{modularfunction}
g(z) = \frac{\eta(z)^8 \eta(4z)^{16}}{\eta(2z)^{24}}
\end{equation}
is a modular function on $\Gamma_0(4)$ having a pole only at the cusp $\frac12$.
At the cusp $\frac12$, it has a simple pole whose residue is $1$.
Then, there is a polynomial $P$ with $\ell$-integral coefficients such that
\[
f_0 - \ell \cdot \theta^{2k_0-m} P(g)
\]
is a holomorphic modular form in $M_{k_0 - \frac m2}$.
Moreover, it is congruent to $f$ modulo $\ell$.
This implies that
\[
\omega_\ell(f) \leq \omega_\ell(f\theta^m) - \frac m2.
\]


(2) Let $m'$ be a positive integer such that the weight of $f \theta^{m'}$ is an integer.
Then, by (1) and Theorem \ref{property_integral} (2), we see that
\[
\omega_{\ell}(f^m) + \frac{mm'}{2} = \omega_{\ell}(f^m \theta^{mm'}) = \omega_{\ell}( (f\theta^{m'})^m) = m\cdot \omega_{\ell}(f \theta^{m'}) = m\cdot \omega_{\ell}(f) + \frac{mm'}{2}.
\]
This gives the desired result.

(3) Since $\ell$ is an odd prime, $f\cdot \theta^\ell$ is a modular form of integral weight on $\Gamma_0(4)$.
Then, there is a modular form $F$ of integral weight $k$ on $\Gamma_0(4)$ with $\ell$-integral rational Fourier coefficients
such that
\[
(f\cdot \theta^\ell) | U_\ell \equiv F \pmod{\ell}.
\]
Note that
\begin{equation} \label{thetacomputation2}
(f\cdot \theta^\ell)|U_\ell \equiv (f|U_\ell) (\theta^\ell | U_\ell) \equiv (f|U_\ell) \theta \pmod{\ell}.
\end{equation}
Here, we used the fact that $\theta^\ell | U_\ell \equiv \theta \pmod{\ell}$.
We define
\[
f_0 = \frac{F}{\theta},
\]
which is a weakly holomorphic modular form with $\ell$-integral rational Fourier coefficients.
By the same argument as in the proof of (1), there is a polynomial $P$ with $\ell$-integral coefficients such that
\[
f_0 - \ell \cdot \theta^{2k-1}P(g)
\]
is a holomorphic modular form in $M_{k-\frac12}$, where $g$ is the modular function defined in (\ref{modularfunction}).
Furthermore, we have
\[
f_0 - \ell \cdot \theta^{2k-1}P(g) \equiv f_0 \equiv \frac{F}{\theta} \equiv f|U_\ell \pmod{\ell}.
\]

Now, we consider the filtration of $f|U_{\ell}$.
By (\ref{thetacomputation2}),
we obtain
\begin{equation*}
\omega_\ell((f\cdot \theta^\ell)|U_\ell) = \omega_\ell((f|U_\ell) \cdot \theta).
\end{equation*}
Therefore, by Theorem \ref{property_integral} (3), we see that
\begin{equation} \label{filtrationcomputation2}
\omega_\ell((f|U_\ell)\cdot \theta) = \omega_{\ell}((f\cdot \theta^\ell)|U_\ell) = \omega_{\ell}(f\cdot \theta^\ell) + t(\ell-1)
\end{equation}
for some integer $t$.
By (1) and (\ref{filtrationcomputation2}), we obtain
\[
\omega_{\ell}(f|U_\ell) + \frac12 = \omega_\ell(f) + \frac{\ell}{2} + t(\ell-1).
\]
From this, we have
\[
\omega_{\ell}(f|U_\ell) - \omega_\ell(f) = \frac{\ell-1}{2}(2t+1).
\]

(4) Let $m$ be an integer such that the weight of $f \theta^m$ is an integer.
Then, by Theorem \ref{property_integral} (1), we see that
\[
\omega_{\ell}(f) - \frac12 \equiv \omega_{\ell}(f\theta^m) - \frac{m+1}{2} \equiv \left(\lambda+\frac{m+1}{2}\right)- \frac{m+1}{2}  \equiv  \lambda \pmod{\ell-1}.
\]
\end{proof}

\subsection{Distribution of Fourier coefficients modulo a prime $\ell$}
In this subsection, we review results related to the distribution of the Fourier coefficients of modular forms of half-integral weight.

In \cite{BO}, Bruinier and Ono studied the Fourier coefficients of a modular form of half-integral weight modulo powers of a prime. They proved the following theorem.
\begin{thm} \cite{BO} \label{AB}
Let $N$ be a positive integer and $\chi$ be a real Dirichlet character modulo $4N$.
Let $\lambda$ be a non-negative integer.
Suppose  that $F$ is a cusp form of weight $\lambda+\frac12$ on $\Gamma_0(4N)$ with character $\chi$ and that its Fourier coefficients are integral.
Let $\ell$ be an odd prime and $j$ be a positive integer.
Then, at least one of the following is true:
\begin{enumerate}
\item The Fourier coefficients of $F$ are well-distributed modulo $\ell^j$.
\item $F$ has  Fourier expansion of the form  (\ref{thetaform}).
\end{enumerate}
\end{thm}

Ahlgren and Boylan \cite{AB} proved  bounds for the weight of a cusp form having  Fourier expansion of the form  (\ref{thetaform}). In the proof, they used the Shimura lift and the theory of Galois representations. In \cite{C}, the first author used only the theory of modular forms modulo $\ell$ to reprove these bounds.
Later, the first author and Kilbourn \cite{CK} improved these bounds.
Since the theory of modular forms modulo $\ell$ can be applied to any holomorphic modular form,  the bounds in  \cite{AB, C, CK}  can be extended to holomorphic modular forms. For the convenience of the reader, we state the result for these bounds.

\begin{thm} \cite{AB, C, CK} \label{ABC}
Suppose that we have the following hypotheses:
\begin{itemize}
\item $\lambda \geq2$ is an integer.

\item $\ell \geq 5$ is a prime.

\item $F$ is a modular form of half-integral weight such that
\[
F(z) = \sum_{n=0}^\infty a(n)q^n \in M_{\lambda+\frac12} \cap \ZZ[[q]].
\]

\item $F \not\equiv 0 \pmod{\ell}$, and there are finitely many square-free integers $n_1, \ldots, n_t$ such that
\[
F(z) \equiv a(0) +  \sum_{i=1}^t \sum_{m=1}^\infty a(n_im^2)q^{n_im^2} \pmod{\ell}.
\]
\end{itemize}
If we write $\lambda = \bar{\lambda} + \iota_{\lambda}(\ell-1)$ with $0\leq \bar{\lambda}\leq \ell-2$,
then the following are true:
\begin{enumerate}
\item[(1)] If $\ell\nmid n_i$ for some $i$, then
\[
\bar{\lambda} \leq 2\iota_{\lambda}+ 1.
\]

\item[(2)] If $\ell | n_i$ for all $i$ and $\bar{\lambda}\leq \frac{\ell-3}{2}$, then
\[
\bar{\lambda} \leq \iota_{\lambda}-\frac{\ell+1}{2}.
\]

\item[(3)] If $\ell | n_i$ for all $i$ and $\bar{\lambda} \geq \frac{\ell-1}{2}$, then
\[
\bar{\lambda} \leq \iota_{\lambda} + \frac{\ell-1}{2}.
\]
\end{enumerate}
\end{thm}

Furthermore, the following lemma was proved in the proof of Proposition 5.1 in \cite{ACR}.

\begin{lem}\cite[Proposition 5.1]{ACR} \label{ACRlemma}
Let $\ell\geq5$ be a prime and $\lambda$ be a non-negative integer.
Let $K$ be a number field, and $\nu$ be a prime ideal of $K$ above $\ell$.
Let $f\in \mathcal{O}_{\nu}[[q]]$ be a cusp form of weight $\lambda+\frac12$ on $\Gamma_0(4)$ with  trivial character.
Suppose that $f$ has  Fourier expansion of the form  (\ref{thetaform}) and $f \not\equiv 0 \pmod{\nu}$.
Then, we have
\[
d_i = 1, \ell, 2,\ \text{or}\ 2\ell.
\]
\end{lem}

For the reader's convenience, we briefly review the proof of Lemma \ref{ACRlemma}.

\begin{proof}
We may assume that  for each $1\leq i\leq t$, there is a positive integer $n_i$ such that $a(d_in_i^2) \not\equiv 0 \pmod{\nu}$.
By the argument in Lemma 4.1 of \cite{AB}, we can find odd primes $p_1, \ldots, p_s$ and a cusp form $g_i\in \mathcal{O}_{\nu}[[q]]$ of weight $\lambda+\frac12$ on $\Gamma_0(4p_1^2\cdots p_s^2)$ with trivial character
such that
\begin{enumerate}
\item $p_1, \ldots, p_s$ are relatively prime to $d_in_i\ell$,
\item $g_i(z) \equiv \sum_{(n, p_1\cdots p_s)=1} a(d_in^2)q^{d_in^2} \not\equiv 0 \pmod{\nu}$.
\end{enumerate}
Then, $g_i^4$ is a cusp form of weight $4\lambda + 2$ on $\Gamma_0(4p_1^2\cdots p_s^2)$ with trivial character and
the Fourier expansion of $g_i^4$ is of the form
\[
g_i^4(z) \equiv \sum_{n=1}^\infty b(d_i n)q^{d_i n} \pmod{\nu}.
\]
By Theorem 3.1 of \cite{ACR}, we have
\[
d_i = 1, \ell, 2,\ \text{or}\ 2\ell.
\]
\end{proof}

\subsection{Mod $\ell$ weakly holomorphic modular forms}

In this subsection, we review results, for a given weakly holomorphic modular form $f$, constructing a holomorphic modular form involving coefficients of $f$ modulo a power of a prime. They play important roles in applying congruence properties for holomorphic modular forms to weakly holomorphic modular forms.

Let $f(z)=\sum_{n\gg-\infty} a(n)q^n$ be a weakly holomorphic modular form of weight $\lambda+\frac{1}{2}$ on $\Gamma_0(N)$ with integral coefficients. Let $\ell \geq 5$ be a prime such that $\ell \nmid N$.
Treneer proved in \cite{T} that
if $m$ is sufficiently large, then for every positive integer $j$, the form $f|U_{\ell^m}$ is congruent to a holomorphic modular form $g$ of weight $\lambda+ \frac{\ell^\beta(\ell^2-1)}{2}+ \frac12$ on $\Gamma_0(N \ell^2)$ modulo $\ell^j$.
The proof of Proposition 5.1 in \cite{AB1} implies that there is a holomorphic modular form $g'$ of weight $\lambda\ell^{a}+\frac{\ell^a-1}{2}+ t(\ell-1)+\frac{1}{2}$ on $\Gamma_0(N)$ which is congruent to $g$ modulo  $\ell$, where $a$ is a positive even integer and $t$ is a positive  integer.
Let us note that the form $g'$ is obtained by considering the image of the product of an Eisenstein series and $g$ under the trace map from $\Gamma_0(N\ell)$ and $\Gamma_0(N)$. Then we have the following theorem.

\begin{thm}\cite[Theorem 3.1]{T}\cite[Proposition 5.1]{AB1}\label{weakholo}
Let $f(z)=\sum_{n\gg-\infty} a(n)q^n$ be a weakly holomorphic modular form of weight $\lambda+\frac{1}{2}$ on $\Gamma_0(N)$ with integral coefficients. Let $\ell \geq 5$ be a prime such that $\ell \nmid N$.
Then, there is a positive integer $m_0$ satisfying the following properties.
\begin{enumerate}
\item If $m\geq m_0$, then we can take a holomorphic modular form $g_m$ of weight $\lambda'+\frac12$ on $\Gamma_0(N)$ such that $\lambda' \equiv \lambda \pmod{\ell-1}$ and
\[
g_m(z) \equiv \sum_{n=0}^{\infty} a(\ell^{m} n)q^n \pmod{\ell}.
\]
\item If $m\geq m_0$, then, for every positive integer $j$, we can take an integer $\beta\geq j-1$ and a cusp form $f_{m,j}$ of weight $\lambda+ \frac{\ell^\beta(\ell^2-1)}{2}+ \frac12$ on $\Gamma_0(N \ell^2)$ with a real Dirichlet character such that
\[
f_{m,j}(z) \equiv \sum_{n=1\atop \ell\nmid n}^{\infty} a(\ell^{m} n)q^n \pmod{\ell^j}.
\]
\end{enumerate}
\end{thm}

\section{Proof of Main Theorems} \label{section3}
In this section, we prove Theorems \ref{main1} and \ref{main2}.

\subsection{Proof of Theorem \ref{main1}}
To prove this, we need the following lemma.

\begin{lem} \label{filtrationcomputation}
Let $\ell\geq5$ be a prime.
Suppose that $f$ is a modular form  with integral coefficients in $M_k$ for $k\in \ZZ$.
If $\omega_\ell(f) > \ell+1$, then
\[
\omega_\ell(f|U_\ell) < \omega_\ell(f).
\]
\end{lem}

\begin{proof}
By Theorem \ref{property_integral} (4), we have
\[
\omega_\ell(f|U_\ell) \leq \ell + \frac{\omega_\ell(f) - 1}{\ell}.
\]
Therefore, if $\omega_\ell(f)>\ell+1$, then we see that
\[
\omega_\ell(f|U_\ell) - \omega_\ell(f) \leq \ell + \frac{\omega_\ell(f) - 1}{\ell} - \omega_\ell(f) = \frac{1}{\ell}(1-\ell)(\omega_\ell(f)-(1+\ell)) < 0.
\]
\end{proof}

Theorem \ref{weakholo} (1) implies that there exist a positive integer $m_1$ and a holomorphic modular form $g$ of weight $k+\frac12$ on $\Gamma_0(4)$ such that
\begin{equation} \label{holomorphicg}
(f|U_{\ell^m})|U_{\ell^{m_1}} \equiv g \pmod{\ell},\ \text{and}\ k \equiv \lambda \pmod{\ell-1}.
\end{equation}
Since $a(0)\not\equiv 0 \pmod{\ell}$, the modular form $g$ is not congruent to $0$ modulo $\ell$.

First, we  show that
\begin{equation} \label{claim}
\omega_\ell(g| U_{\ell^{t}}\cdot \theta^\ell) \leq \ell+1
\end{equation}
for some non-negative integer $t$.
If $\omega_\ell(g\cdot \theta^\ell)\leq  \ell+1$, then (\ref{claim}) is satisfied for $t=0$.
Assume that  $\omega_\ell(g\cdot\theta^\ell) > \ell+1$.
By Lemma \ref{filtrationcomputation}, we obtain
\begin{equation*} \label{modlcomputation1}
\omega_\ell((g\cdot\theta^\ell)|U_\ell) < \omega_\ell(g\cdot \theta^\ell).
\end{equation*}
From this and Theorem \ref{property_integral}, we see that
\[
 \omega_{\ell}((g\cdot \theta^{\ell})|U_{\ell})=\omega_{\ell}(g\cdot \theta^{\ell})  - t(\ell-1)
\]
for a positive integer $t$.
Therefore, we have
\begin{equation} \label{U_inequality}
\omega_{\ell}((g\cdot \theta^{\ell})|U_{\ell}) \leq \omega_{\ell}(g\cdot \theta^{\ell}) - (\ell-1).
\end{equation}
By the same argument in (\ref{thetacomputation2}), we see that
\begin{equation} \label{thetacomputation}
\omega_\ell((g\cdot \theta^\ell)|U_\ell) = \omega_\ell((g|U_\ell) \cdot \theta).
\end{equation}
By  Proposition \ref{filtration}(1), (\ref{thetacomputation}),  and (\ref{U_inequality}), we have
\begin{eqnarray*}
\omega_\ell((g|U_\ell) \cdot \theta^\ell) &=&
\omega_\ell(g|U_\ell \cdot  \theta) + \frac{\ell-1}{2}= \omega_{\ell}((g\cdot \theta^\ell)|U_\ell) + \frac{\ell-1}{2} < \omega_\ell(g\cdot\theta^\ell).
\end{eqnarray*}
By repeating the above process, we see that  there is a positive integer $t$ such that
\[
\omega_\ell((g| U_{\ell^{t}})\cdot \theta^\ell) \leq \ell+1.
\]

Note that (\ref{claim})  implies
\begin{equation} \label{inequality}
\omega_\ell(g|U_{\ell^t}) \leq \ell+1 - \frac{\ell}{2} = \frac{\ell}{2} + 1.
\end{equation}
Since we have
\[
g|U_{\ell^t} \equiv (f|U_{\ell^{m+m_1}})|U_{\ell^{t}} \pmod{\ell},
\]
by the hypothesis (\ref{thetaform}), we see that the Fourier expansion of $g|U_{\ell^t}$ has the form (\ref{thetaform}).
The inequality (\ref{inequality}) implies that, in the notation of Theorem \ref{ABC}, we have $\bar{\lambda} = \lambda$ and $\iota_{\lambda} = 0$.
Thus, $\lambda\leq1$ or $\lambda = \frac{\ell-1}{2}$, so
 $\omega_\ell(g|U_{\ell^t}) \in \{\frac12, \frac32, \frac{\ell}{2}\}$.

If $\omega_\ell(g|U_{\ell^t}) = \frac 32$, then
\begin{equation} \label{gtheta}
g|U_{\ell^t} \equiv c \cdot \theta^3\ (\mathrm{mod}\ \ell)
\end{equation}
for some $\ell$-integral rational constant $c \not\equiv 0 \pmod{\ell}$, since $M_{3/2}$ is a one-dimensional vector space containing $\theta^3$.
We consider the theta operator defined by
\begin{equation} \label{theta}
\Theta\left(\sum_{n=0}^\infty a(n)q^n\right) = \sum_{n=0}^\infty na(n)q^n.
\end{equation}
Let $E_{k}$ be the normalized Eisenstein series of weight $k$. Thus,
$$(\ell-1) \Theta(\theta^3) E_{\ell-1} - \frac{3}{2} \Theta( E_{\ell-1})\theta^3$$ is a cusp form of weight $\ell+1+\frac{3}{2}$ on $\Gamma_0(4)$ (see Corollary 7.2 in \cite{Co}). Let us note that, for each $\ell$,  the Fourier coefficients of $E_{\ell-1}$ are $\ell$-integral and $E_{\ell-1} \equiv 1 \pmod{\ell}$ (for example, see Theorem 7.1 in Chapter X of \cite{Lang}). Thus, there is a cusp form $h$ of weight $\ell+1+\frac{3}{2}$ on $\Gamma_0(4)$ such that
\[
\Theta(\theta^3) \equiv h \pmod{\ell}.
\]
Note that, by (\ref{gtheta}), $h$ has Fourier expansion of the form (\ref{thetaform}).
By a computation, we have
\[
h(z) \equiv \Theta(\theta^3)(z) \equiv  6q + 24q^2 + 24q^3 + \cdots \pmod{\ell}.
\]
Since $24\not\equiv 0 \pmod{\ell}$ for $\ell\geq5$, we see that one of the $d_i$ in (\ref{thetaform}) is $3$.
This is a contradiction due to Lemma \ref{ACRlemma}.
Therefore, $\omega_\ell(g|U_{\ell^t})$ is $\frac 12$ or $\frac{\ell}{2}$.
By (\ref{holomorphicg}) and  Proposition \ref{filtration} (4), we see that
\[
\lambda \equiv \omega_\ell(f)-\frac12 \equiv
\omega_{\ell}(f|U_{\ell^{m+m_1}}) -\frac12 \equiv \omega_{\ell}(g) -\frac12 \equiv \omega_{\ell}(g|U_{\ell^t}) -\frac12 \equiv 0 \pmod{\frac{\ell-1}{2}}.
\]

Now, we  show that there is a positive integer $s$ such that
\[
f|U_{\ell^s} \equiv  a(0)\theta \pmod{\ell}.
\]
If $\omega_{\ell}(g|U_{\ell^t}) = \frac12$, then
\[
f|U_{\ell^{m+m_1+t}} \equiv      g|U_{\ell^t} \equiv a(0)\theta \pmod{\ell},
\]
since $M_{\frac12}$ is a one-dimensional space spanned by $\theta$.
Suppose that $\omega_{\ell}(g|U_{\ell^t}) = \frac{\ell}{2}$.
Then, $g|U_{\ell^t}$ has  Fourier expansion of the form  (\ref{thetaform}) and $\ell | d_i$ for all $i$  by Theorem \ref{ABC}.
This implies that
\[
\{(g|U_{\ell^t})|U_{\ell}\}^\ell \equiv g|U_{\ell^t} \pmod{\ell}.
\]
By Proposition \ref{filtration} (2), we see that
\[
\frac{\ell}{2} = \omega_{\ell}(g|U_{\ell^t}) =
\omega_\ell(\{(g|U_{\ell^t})|U_{\ell}\}^\ell) = \ell \cdot \omega_{\ell}((g|U_{\ell^t})|U_{\ell}).
\]
From this, we obtain
\[
\omega_{\ell}((g|U_{\ell^t})|U_{\ell}) = \frac12,
\]
and hence, we have
\[
f|U_{\ell^{m+m_1+t+1}} \equiv    (g|U_{\ell^t})|U_{\ell} \equiv a(0)\theta \pmod{\ell}.
\]

Since we have
\[
\theta|U_\ell \equiv \theta^\ell \equiv 1 + 2\sum_{n=1}^{\infty} q^{\ell n^2} \pmod{\ell}
\]
and
\[
(\theta^\ell)|U_{\ell} \equiv \theta \pmod{\ell},
\]
we see that
\[
\biggl\{ \lim_{n\to\infty} \tilde{f}|U_{\ell^{2n}}, \lim_{n\to\infty} \tilde{f}|U_{\ell^{2n+1}} \biggr\} = \biggl\{ a(0)\tilde{\theta}, a(0)\tilde{\theta}^{\ell}\biggr\}.
\]

\subsection{Proof of Theorem \ref{main2}}
Assume 
 that $f(z) = \sum_{n\gg-\infty} a(n)q^n$ is a weakly holomorphic modular form in $M^!_{\lambda+\frac12}$ with integral coefficients
such that $\lambda\not\equiv 0 \pmod{\frac{\ell-1}{2}}$.
Let us fix a positive integer $j$.
Theorem \ref{weakholo} implies that there exists a positive integer $m_0$ such that, for $m\geq m_0$,
we can take a cusp form $f_{m,j}$ on $\Gamma_0(4\ell^2)$ satisfying
\begin{equation} \label{gmj}
f_{m,j}(z) \equiv (f|U_{\ell^m} - f|U_{\ell^{m+1}}|V_{\ell})(z) \equiv
 \sum_{n=1\atop \ell\nmid n}^\infty a(\ell^m n )q^n \pmod{\ell^j},
\end{equation}
where the operator $V_\ell$  on formal power series  is defined by
\[
\left( \sum_{n=0}^\infty a(n)q^n \right) \bigg| V_\ell = \sum_{n=0}^\infty a(n)q^{\ell n}.
\]
Furthermore, there is a holomorphic modular form $g(z) = \sum_{n=0}^\infty b(n)q^n$ on $\Gamma_0(4)$ such that
\begin{equation} \label{hmfg}
g(z) \equiv (f|U_{\ell^{m_0}})(z) \equiv \sum_{n=0}^\infty a(\ell^{m_0}n)q^n \pmod{\ell}.
\end{equation}
To prove Theorem \ref{main2}, we need the following lemma.

\begin{lem} \label{family}
With the above notation, there is a positive integer $M$ such that the Fourier expansion of $f_{M,j}$ is not of the form  (\ref{thetaform}).
\end{lem}

\begin{proof}
Suppose that, for every positive integer $m\geq m_0$, the Fourier expansion of $f_{m,j}$ is of the form  (\ref{thetaform}), i.e.,
$(f_{m,j})(z) = \sum_{n=1}^\infty c_{m,j}(n)q^n$ has  Fourier expansion of the form
\[
f_{m,j}(z) \equiv \sum_{i=1}^{t(m)} \sum_{n=1}^\infty c_{m,j}(d(m,i)n^2)q^{d(m,i) n^2} \pmod{\ell},
\]
where $d(m,i)$ are square-free positive integers.

Let
\[
e(m,i)=
\begin{cases}
d(m,i) & \text{if $m-m_0$ is even},\\
d(m,i)\ell & \text{if $m-m_0$ is odd}.
\end{cases}
\]
Note that, by (\ref{gmj}) and (\ref{hmfg}), we have
\[
c_{m,j}(d(m,i)n^2) \equiv b(\ell^{m-m_0} d(m,i) n^2) \pmod{\ell}
\]
and that
$e(m,i)$ is the square-free part of $\ell^m d(m,i) n^2$ since $\ell \nmid d(m,i)$.

Since $\lambda \not\equiv 0 \pmod{\frac{\ell-1}2}$, by (\ref{hmfg}) and Theorem \ref{main1}, we see that the Fourier expansion of
$g$
is not of the form  (\ref{thetaform}).
This means that there are infinitely many square-free positive integers $d$ such that
\[
b(dn^2) \not\equiv 0 \pmod{\ell}
\]
for some $n\in\ZZ$.
Let $T$ be the set of such square-free positive integers $d$.

Let $m_1\geq m_0+1$ be a positive integer.
Let
\[
S = \{ e(m,i)\ | m < m_1\}.
\]
This is a finite set.
Then, there is a square free integer $t \in T - S$ since $T$ is an infinite set.
By the argument in Lemma 4.1 of \cite{AB}, we can construct a modular form $h(z) =\sum_{n=0}^{\infty}c(n)q^n$ of weight $k+\frac{1}{2}$ on $\Gamma_0(4N')$ from $g$ such that
\begin{enumerate}
\item $N'$ is relatively prime to $\ell$,
\item $h \not \equiv 0 \pmod{\ell}$,
\item $c(n) \equiv 0 \pmod{\ell}$ if $e(m,i)|n$ for some $e(m,i) \in S$.
\end{enumerate}
Then, the Fourier expansion of $h$ is of the form
\[
h(z) \equiv \sum_{n=0}^\infty c(\ell^{m'}n) q^{\ell^{m'}n} \pmod{\ell},
\]
where $m' = m_1-m_0$.
Therefore, we have
\[
(h|U_{\ell^{m'}})^{\ell^{m'}} \equiv h \pmod{\ell}.
\]
This means that
\[
\omega_\ell(g^2) \geq  \omega_\ell(h^2) = \omega_\ell((h|U_{\ell^{m'}})^{2\ell^{m'}}) = \ell^{m'} \cdot \omega_\ell((h|U_{\ell^{m'}})^2) \geq \ell^{m'}.
\]
This is a contradiction since $m'$ can be any positive integer with $m'\geq1$.
\end{proof}

By Lemma \ref{family}, there is a positive integer $M$ such that the Fourier expansion of $f_{M,j}$ is not of the form  (\ref{thetaform}).
Let $r$ be an integer with $r \not \equiv 0 \pmod{\ell^j}$.
By Theorem \ref{AB}, the coefficients of $f_{M,j}$ are well-distributed modulo $\ell^j$.
Therefore,   we have
\begin{equation*}
\#\{ 0 \leq n \leq X \; | \;  c_{M,j}(n)\equiv r\ (\mathrm{mod}\ \ell^j) \} \gg_{r,\ell^j}
{\sqrt{X}}/{\log X},
\end{equation*}
where $c_{M,j}(n)$ denotes the $n$th Fourier coefficient of $f_{M,j}$.
By (\ref{gmj}), we obtain
\begin{equation*}
\#\{ 0 \leq n \leq \ell^M X \; | \;  a(n)\equiv r\ (\mathrm{mod}\ \ell^j) \} \gg_{r,\ell^j}
{\sqrt{X}}/{\log X}.
\end{equation*}
Therefore, we have
\begin{equation*}
\#\{ 0 \leq n \leq X \; | \;  a(n)\equiv r\ (\mathrm{mod}\ \ell^j) \} \gg_{r,\ell^j}
{\sqrt{X/\ell^M}}/{\log (X/\ell^M)} \gg_{r,\ell^j}{\sqrt{X}}/{\log X}.
\end{equation*}
Furthermore, in \cite{T}, Treneer  proved that
\begin{equation*}
\#\{ 0 \leq n \leq X \; | \;  a(n)\equiv 0\ (\mathrm{mod}\ \ell^j) \} \gg_{r,\ell^j} X.
\end{equation*}
Therefore, the Fourier coefficients of $f$ are well-distributed modulo $\ell^j$.

\section*{Acknowledgments}
The authors are grateful to the referee for useful comments. The authors also thank Scott Ahlgren for helpful comments on the previous version of this paper.

\end{document}